\documentclass[11pt,letterpaper]{article}

\usepackage{amsmath}        
\usepackage{amssymb}        
\usepackage{amsthm}     

\usepackage{color,graphicx}    
\usepackage{enumerate}      
\usepackage[normalem]{ulem}    

\newtheorem{theorem}{Theorem}
\newtheorem{proposition}{Proposition}[section]
 \newtheorem{definition}{Definition}
\newtheorem{lemma}[proposition]{Lemma}

\newcommand\grad{{\bf \nabla}}
\newcommand\rvec{{\bf r}}
\newcommand\la{{\lambda}}

\newcommand\nvec{{\bf n}}
\newcommand\xvec{{\bf x}}

\newcommand\mvec{{\bf m}}

\newcommand\eps{{\epsilon}}
\newcommand\Qvec{{\bf Q}}

\newcommand \I {{\bf I_{LG}}}

\newcommand{\Acal}{{\cal A}}

\newcommand{\Rr}{{\mathbb R}}

\newcommand{\dd}{\, \mathrm{d}}

\DeclareMathOperator{\dist}{dist}       

\DeclareMathOperator{\loc}{loc}
\newcommand{\N}{\mathbb{N}}             

\newcommand{\R}{\mathbb{R}}             

\DeclareMathOperator{\sgn}{sgn}

\DeclareMathOperator{\tr}{tr}
\renewcommand{\S}{\mathbb{S}}
\newcommand{\vbar}[1]{\bar{\vec{#1}}}
\renewcommand{\vec}{\mathbf}    

\newcommand{\vtilde}[1]{\tilde{\vec{#1}}}

\begin{document}

\centerline{\large \bf Uniaxial versus Biaxial Character of Nematic Equilibria in Three Dimensions}
\vspace{0.6em}
\centerline{\today}

\vspace{1.5em}
\begin{center}
{Apala Majumdar\\
 Department of Mathematical Sciences,
University of Bath, Bath, BA2 7AY, United Kingdom \\}
\end{center}

\begin{center}{Adriano Pisante\\
 Dipartimento di Matematica ``G.~Castelnuovo'', Sapienza, Universit\`a di Roma,
 00185 Rome, Italy
\\} \end{center}

\begin{center} {Duvan Henao\\
 Facultad de Matem\'aticas, Pontificia Universidad Cat\'olica de Chile,
 Casilla 306, Correo 22, Santiago, Chile
\\}\end{center}

\begin{abstract}

We study global minimizers of the Landau-de Gennes (LdG) energy
functional for nematic liquid crystals, on arbitrary
three-dimensional simply connected geometries with topologically
non-trivial and physically relevant Dirichlet boundary conditions.
Our results are specific to the low-temperature limit. We prove
(i) that (re-scaled) global LdG minimizers converge uniformly to a
(minimizing) limiting harmonic map, away from the singular set of
the limiting map; 
(ii) there exist
both a point of maximal biaxiality
and a nonempty Lebesgue-null set of uniaxial points
near each singular point of the limiting harmonic map
(this improves the recent results of \cite{contreraslamy});
 (iii)
estimates for the size of ``strongly biaxial" regions in terms of
the reduced temperature $t$. 
We further show that global LdG minimizers in the restricted class of uniaxial $\Qvec$-tensors cannot be stable critical points of the LdG energy for low temperatures.  

\end{abstract}
\section{Introduction}
\label{sec:intro}

Nematic liquid crystals (LCs) are anisotropic liquids with
long-range orientational order i.e. the constituent rod-like
molecules have full translational freedom but align along certain
locally preferred directions \cite{dg,virga}. The existence of
distinguished directions renders nematics sensitive to light and
external fields leading to unique electromagnetic, optical and
rheological properties \cite{dg,lavrentovich,zbd}. The
analysis of nematic spatio-temporal patterns is a fascinating
source of problems for mathematicians, physicists and engineers
alike, especially in the context of defects or material
singularities \cite{lavrentovich, vitelli}.

In recent years, mathematicians have turned to the analysis of the
celebrated Landau-de Gennes (LdG) theory for nematic liquid
crystals, particularly in two asymptotic limits: the vanishing
elastic constant and the low-temperature limit; see for example
\cite{amaz, hm, canevari, contreraslamy} which is not an
exhaustive list but are directly relevant to our paper. The LdG
theory is a variational theory with an associated energy
functional, defined in terms of a macroscopic order parameter,
known as the $\Qvec$-tensor order parameter \cite{dg,virga,lin}.
The LdG energy typically comprises an elastic energy, convex in
$\grad \Qvec$ with several elastic constants, and a non-convex
bulk potential, $f_B$ defined in terms of the temperature and the
eigenvalues of $\Qvec$-tensor \cite{newtonmottram}. With the
one-constant approximation for the elastic energy density, the LdG
energy functional has a similar structure to the Ginzburg-Landau
(GL) functional for superconductivity \cite{bbh,PacardRiviere,
chapman} and in certain asymptotic limits (limit of vanishing
elastic constant and low-temperature limit), we can borrow several
ideas from GL theory to make qualitative predictions about global
energy minimizers, at least away from singularities. However,
there is an important distinction between the GL theory and LdG
theory. In the GL-framework, researchers study maps, $\mathbf{u}:
\Rr^d \to \Rr^d$, $d=2,3$ (see e.g. \cite{bbh, pisante, MillotPisante}), whereas the
LdG variable is a five-dimensional map, $\Qvec: \Rr^3 \to \Rr^5$.
A uniaxial $\Qvec$-tensor has three degrees of freedom (with an
order parameter and a single distinguished direction) and in the
uniaxial case, there is broader scope for methodology transfer
from GL-based techniques (see for example, \cite{hm}).
 A biaxial $\Qvec$-tensor has five degrees of freedom and there are
a plethora of open questions about how the two extra degrees of
freedom manifest in the mathematics and physics of biaxial
systems.

We re-visit questions related to the uniaxial versus biaxial
structure of global LdG minimizers in the low-temperature limit.
We work with ``nice" three-dimensional (3D) domains as described in
the abstract and with fixed topologically non-trivial Dirichlet
boundary conditions. In particular, the Dirichlet condition is a
minimizer of the potential, $f_B$, in the LdG energy. Our first
result concerns the uniform convergence of global energy
minimizers to a (minimizing) limiting harmonic map, away from the
singular set of the limiting harmonic map, in the low-temperature
limit. The uniform convergence follows from a Bochner-type
inequality for the LdG energy density, first derived in
\cite{amaz} in the vanishing elastic constant limit. There are
subtle mathematical differences between the low-temperature and
vanishing elastic constant limits; in particular, the bulk potential $f_B$ comprises two terms that diverge at different rates in the low temperature limit. This is explained in more detail 
after the proof of Lemma \ref{le:bochner}
 and requires us to consider mathematical cases or
scenarios which do not arise in the vanishing elastic constant
case. The uniform convergence gives a fairly good description of
global energy minimizers away from the singular set, $\Sigma$, of
the limiting harmonic map. The singular set, $\Sigma$, consists of
a discrete set of point defects. In \cite{contreraslamy}, the
authors prove the existence of a point of maximal biaxiality for
global minimizers, in the low-temperature limit but do not comment
on the number of such points. We appeal to a
topological result in \cite{canevari} to deduce the existence of
a point of maximal biaxiality and a point of uniaxiality near each singular point in
$\Sigma$, in the low-temperature limit. Maximal biaxiality is
understood in terms of the biaxiality parameter which varies
between $0$ and unity (see \cite{virga, g2} and subsequent
sections for a definition of the biaxiality parameter). We make
the notion of ``strongly biaxial" regions in global energy
minimizers more precise by computing estimates for the size of
such regions in terms of the reduced temperature, $t$. The proof
depends on scaling and blow-up arguments and well-established
results in GL-theory and the theory of harmonic maps (e.g.
\cite{bcl, schoen}). We consider all admissible scenarios and
exclude all but one scenario, based on the arguments above and find that the size of 
``strongly biaxial" regions scales as $t^{-1/4}$ as $t \to \infty$.

In \cite{hm}, we study global LdG energy minimizers on a 3D
droplet with radial boundary conditions, in the low temperature
limit. We appeal to GL-based techniques (see \cite{pisante,
mironescu}) to show that global minimizers, if uniaxial, must have
the radial-hedgehog (RH) structure for low temperatures. The RH
solution is a radially-symmetric critical point of the LdG energy
on a 3D droplet, with a single isotropic point (with $\Qvec=0$) at
the droplet centre and perfect uniaxial symmetry away from the
centre i.e.\ the molecules point radially outwards everywhere away
from the centre \cite{g2,hm}. Further, it is known that the
RH-solution is unstable with respect to symmetry-breaking biaxial
higher-dimensional perturbations for low temperatures
\cite{g2,am2}. In \cite{hm}, it is shown that global minimizers, in
the restricted class of uniaxial tensors, cannot be stable
critical points of the LdG energy on a 3D droplet, for low
temperatures, since they converge to the RH solution in the
low-temperature limit. 
 In
\cite{lamyuniaxial}, the author uses symmetry-based arguments and the structure of the LdG Euler-Lagrange equations to prove that the radial-hedgehog solution (modulo a rotation) is the unique uniaxial critical point on a 3D droplet with radial boundary conditions, for all temperatures.

Our second theorem in this paper generalizes the results in
\cite{hm} to 
arbitrary three-dimensional
geometries with arbitrary physically relevant topologically
non-trivial Dirichlet conditions. There exists a global LdG energy
minimizer in the restricted class of uniaxial $\Qvec$-tensors, for
all temperatures. These restricted uniaxial minimizers necessarily
have non-negative scalar order parameter and satisfy a physically
relevant energy bound. We show that these restricted minimizers
cannot be stable critical points of the LdG energy for low
temperatures. The argument proceeds by contradiction. Appealing to
topological arguments, we show that any uniaxial critical point of
the LdG energy has an isotropic point near each singular point of
the limiting harmonic map, for sufficiently low-temperatures. We
then proceed with a local version of the global analysis in
\cite{hm}, equipped with certain energy quantization results for
harmonic maps \cite{bcl} and blow-up techniques, to deduce the
local RH-structure near each isotropic point. In other words, we
reduce the local analysis near an isotropic or a defect point (in
the uniaxial case) to the model problem of uniaxial equilibria on
a 3D droplet with radial boundary conditions and the local
instability of the RH-profile suffices for our purposes. 
In fact, we believe that the universal RH-defect profile for
uniaxial critical points is not specific to the low-temperature
limit but also extends to the vanishing elastic constant limit
(for all $t>0$) and indeed any limit for which we are guaranteed
uniform convergence to a (minimizing) limiting harmonic map, away
from the singularities of the limiting harmonic map. It is known from \cite{bcl}
that point defects for minimizing harmonic maps have a local
radial profile (modulo a rotation) and we believe that the
uniaxial defect profiles have the same radial profile, weighted by
an appropriate scalar order parameter, which is precisely the RH
solution. The analysis of defect profiles near points of maximal
biaxiality remains an open problem for analysts, although some key
steps are given in \cite{sluckin, penzenstadler, kraljvirga}.

The paper is organized as follows. In Section~\ref{sec:2}, we
review the theoretical background and state our main results. In
Section~\ref{sec:3}, we give the proofs and in
Section~\ref{sec:4}, we conclude with future perspectives.

\section{Statement of results}
\label{sec:2}

Let $\Omega \subset \Rr^3$ be an arbitrary simply-connected 3D domain with smooth boundary. Let $\S^2$ be the set of unit vectors in $\R^3$ and let
$S_0$ denote the set of symmetric, traceless
$3\times 3$ matrices i.e.
\begin{equation}
\label{eq:1} S_0 = \left\{\Qvec\in M^{3\times 3};
\Qvec_{ij}=\Qvec_{ji}; \Qvec_{ii}=0 \right\},
\end{equation}
where $M^{3\times 3}$ is the set of $3\times 3$ matrices. The
corresponding matrix norm is defined to be \cite{amaz}
\begin{equation}
\label{eq:2} |\Qvec|^2 = \Qvec_{ij} \Qvec_{ij} \quad i,j=1\ldots 3
\end{equation}
and we use the Einstein summation convention throughout the paper.

We work with the Landau-de Gennes (LdG) theory for nematic liquid
crystals \cite{dg} whereby a LC state is
described by a macroscopic order parameter: the
$\Qvec$-tensor order parameter. The $\Qvec$-tensor is a macroscopic measure of the LC anisotropy. Mathematically, the LdG $\Qvec$-tensor order parameter is a symmetric, traceless
$3\times 3$ matrix in the space $S_0$ in (\ref{eq:1}).
A LC state is said to be (i) isotropic (disordered with no orientational ordering) when $\Qvec=0$, (ii) uniaxial when $\Qvec$ has two degenerate non-zero eigenvalues and (iii) biaxial when $\Qvec$ has three distinct eigenvalues. A uniaxial $\Qvec$-tensor can be written as
\begin{align} \label{eq:uniaxial}
        \vec Q(\vec x) = s(\vec x) \left (\vec n(\vec x) \otimes \vec n(\vec x) - \frac{\vec I}{3}\right ),
    \end{align}
    for some $s(\vec x)\in \R$ and some unit-vector
    $\vec n(\vec x)\in \S^2$, for a.e.\ $\vec x \in \Omega$. We include $s=0$ in our definition although $s=0$ corresponds to the isotropic phase.
    The unit-vector, $\vec n$, is the director or equivalently, the single
    distinguished direction of molecular alignment in the sense that all
    directions orthogonal to $\vec n$ are physically equivalent for an uniaxial nematic.
    We recall the
definition of the biaxiality parameter \cite{g2}, \cite{virga},
\begin{equation}
\label{eq:maj9} \beta^2 = 1 - \frac{6
(\textrm{tr}\Qvec^3)^2}{|\Qvec|^6} \in \left[0, 1
\right].\end{equation} In particular, $\beta^2=0$ if and only if
$\Qvec$ is uniaxial i.e.\ if and only if $|\Qvec|^6 = 6\left(\textrm{tr} \Qvec^3 \right)^2$.

The LdG theory is a variational theory and has an associated LdG free energy. The LdG energy density is a nonlinear function of $\Qvec$ and its spatial derivatives
\cite{dg,newtonmottram}. We work with the simplest form of the
LdG energy functional that allows for a first-order
nematic-isotropic phase transition and spatial inhomogeneities as
shown below \cite{amaz, newtonmottram}:
\begin{equation}
\label{eq:3} \I \left[\Qvec \right] = \int_{\Omega}
\frac{L}{2}|\grad \Qvec|^2 + f_B\left(\Qvec\right)~dV.
\end{equation}
Here, $L>0$ is a small material-dependent elastic constant,
$|\grad \Qvec|^2 = \Qvec_{ij,k}\Qvec_{ij,k}$ (note that
$\Qvec_{ij,k} = \frac{\partial \Qvec_{ij}}{\partial \xvec_k}$)
with $i,j,k=1,2,3$ is an \emph{elastic energy density} and
$f_B:S_0 \to \Rr$ is the \emph{bulk energy density} that dictates
the preferred nematic phase: isotropic/uniaxial/biaxial. For our
purposes, we take $f_B$ to be a quartic polynomial in the
$\Qvec$-tensor invariants:
\begin{equation}
\label{eq:4} f_B(\Qvec) = \frac{A(T)}{2}\textrm{tr}\Qvec^2 - \frac{B}{3}\textrm{tr}\Qvec^3 +
\frac{C}{4}\left(\textrm{tr}\Qvec^2\right)^2
\end{equation}
where $\textrm{tr}\Qvec^3 = \Qvec_{ij}\Qvec_{jp}\Qvec_{pi}$ with
$i,j,p=1,2,3$; $A(T) = \alpha(T - T^*)$; $\alpha, B, C>0$ are material-dependent
constants, $T$ is the absolute temperature and $T^*$ is a
characteristic temperature below which the isotropic phase,
$\Qvec=0$, loses its stability \cite{newtonmottram,ejam}. We work in the low temperature
regime with $T<<T^*$ (or $A <0$)  and subsequently investigate the $A\to - \infty$ limit,
known as the \emph{low temperature} limit. One can readily verify
that $f_B$ is bounded from below and attains its minimum on the
set of $\Qvec$-tensors given by \cite{ejam,am2}
\begin{equation}
\label{eq:6} \Qvec_{\min} = \left\{ \Qvec \in S_0; \Qvec = s_+
\left(\nvec \otimes \nvec - \frac{\mathbf{I}}{3}\right),~\nvec \in
\S^2 \right\},
\end{equation}
$\mathbf{I}$ is the $3\times 3$ identity matrix and
\begin{equation}
\label{eq:s+} s_+ = \frac{B + \sqrt{B^2 + 24 |A| C}}{4 C}.
\end{equation} The set, (\ref{eq:6}), is the set of uniaxial $\vec Q$-tensors with constant order parameter, $s_+$.

In what follows, we take the Dirichlet boundary condition to be
\begin{align}
\label{eq:DC} \Qvec_{b,A}(\xvec) = s_+ \left( \vec n_b \otimes
\vec n_b - \frac{\mathbf{I}}{3}\right)
\end{align}
for some arbitrary smooth unit-vector field, $\vec n_b$, with topological degree $d\ne 0$
(see, e.g., \cite{degree} and \cite{bcl} for the definition and the main properties of the
topological degree).  
  The corresponding admissible space is
\begin{equation}
\label{eq:AS} \Acal_A = \left\{ \Qvec \in
W^{1,2}\left(\Omega;S_0\right); \Qvec = \Qvec_{b,A} ~\text{on}~
\partial \Omega \right\},
\end{equation}
where $W^{1,2}\left(\Omega;S_0\right)$ is the Soboblev space of
square-integrable $\Qvec$-tensors with square-integrable first
derivatives \cite{evans}, with norm
$$||\Qvec||_{W^{1,2}} = \left(\int_{\Omega} |\Qvec|^2 +
|\grad \Qvec|^2~dV \right)^{1/2}.$$  In what follows, we identify
the degree of a uniaxial $\Qvec$-tensor in $\Acal_A$ on the boundary, with the
degree of the director field, $\nvec \in W^{1,2}\left(\Omega; S^2
\right)$ on the boundary, $\textrm{deg}\, (\nvec, \partial \Omega)$,  which is well defined because $\nvec_b$ is smooth. The existence of a
global minimizer of $\I$ in the admissible space, $\Acal_A$, is
immediate from the direct method in the calculus of variations
\cite{evans}; the details are omitted for brevity. It follows from
standard arguments in elliptic regularity that all global
minimizers are smooth and real analytic solutions of the
Euler-Lagrange equations associated with $\I$ on $\Omega$,
\begin{equation}
\label{eq:ELeqs} L \Delta \Qvec= A \Qvec -
B\left(\Qvec^2 -
\left(\textrm{tr}\Qvec^2\right)\frac{\vec I}{3}\right) + C
\left(\textrm{tr}\Qvec^2\right) \Qvec,
\end{equation}
where $B\left(\textrm{tr}\Qvec^2\right)\frac{\vec I}{3}$ is a
Lagrange multiplier accounting for the tracelessness constraint
\cite{amaz}.

Define the re-scaled maps, $\bar {\vec Q}:= \frac{1}{s_+}\sqrt{\frac{3}{2}} \vec Q$. Let
    \begin{align} \label{eq:rTj}
        t:= \frac{27|A|C}{B^2}, \  h_+=\frac{3+\sqrt{9+8t}}{4},
        \  
        \xi_b = \sqrt{\frac{27 L C}{B^2 t}}
        \quad \text{and}\quad
        \bar L:= \frac{27C}{2B^2}L.
    \end{align} Then
     $  s_+ = \frac{B}{3C} h_+  $
    and the minimum of the bulk energy density, $f_B$ in (\ref{eq:4}), is
    \begin{equation}\label{eq:fBmin}
    \min_{\Qvec \in S_0} f_B\left( \Qvec \right) = - \frac{1}{8}\left( t + h_+ \right).
    \end{equation}
    The low-temperature limit corresponds to $t \to \infty$.

    The re-scaled LdG energy is then given by
    \begin{equation}
    \label{LdGnew2}
            \frac{3 \bar L }{2Ls_+^2}       \mathbf{I}_{LG}[\bar{\Qvec}]
        = \int_{\Omega}
        \frac{\bar L}{2}|\grad \bar{\Qvec}|^2
        + \frac{t}{8}\left(1- |\bar{\Qvec}|^2 \right)^2 +
        \frac{h_+}{8} (1+3|\vbar Q|^4-4\sqrt{6}\textrm{tr}\vbar Q^3)~dV
        \end{equation}
        
    Note that the bulk potential has two contributions in \eqref{LdGnew2}, both of which are both nonnegative in view of \eqref{eq:maj9}. The first term vanishes for $\bar{\Qvec} \in \mathbb{S}^4$ and the second term vanishes if and only if $\bar{\Qvec}$ is uniaxial with unit norm (from \eqref{eq:maj9} again).
            The re-scaled boundary condition is
        $\bar{\vec Q}_{b} = \sqrt{\frac{3}{2}}\left(\nvec_b \otimes \nvec_b - \frac{\mathbf{I}}{3} \right)$. \textbf{In what follows, all statements are to be understood in terms of the re-scaled variables and we drop the bars from the variables for brevity. We recall the definition of a minimizing limiting harmonic map.}


%
\begin{definition}
    A (minimizing) limiting harmonic map with respect to the re-scaled Dirichlet condition in \eqref{eq:DC}, is a uniaxial map of the form
\begin{equation}
    \label{eq:form}
     \Qvec^0 = \sqrt{\frac{3}{2}} \left( \nvec^0 \otimes \nvec^0 -
    \frac{\mathbf{I}}{3} \right),
\end{equation}
where $\nvec^0$ is a minimizer of the Dirichlet energy
\begin{equation}
\label{eq:OFenergy} I[\nvec] = \int_{\Omega} |\grad \nvec|^2~dV
\end{equation}
in the admissible space $\Acal_{\nvec_b} = \left\{ \nvec \in
W^{1,2}\left(\Omega;\S^2\right); \nvec = \vec n_b
~on~ \partial \Omega\right\}$ \cite{schoen}.
\end{definition}
In particular, $\nvec_0$ is a harmonic map into $\S^2$, i.e., a
solution of the harmonic map equations $\Delta \nvec + |\grad
\nvec|^2 \nvec =0$. The singular set of $\nvec_0$, denoted by
$\Sigma=\left\{ \xvec_1\ldots \xvec_N \right\} \subset \Omega$, is a finite set of points \cite{schoen,schoen3}.

We have two main theorems.

%
%
%
%
\begin{theorem} \label{th:biaxial2}

Let $\Omega\subset \R^3$ be as above. Let $\{t_j\}_{j\in\N}$ with
$t_j\overset{j\to\infty}{\longrightarrow} +\infty$ and let $\{\vec
Q_j\}_{j\in \N}$ be a corresponding global minimizer of the LdG
energy in (\ref{LdGnew2}), in the admissible space $\bar{\Acal}_A = \left\{ \Qvec \in
W^{1,2}\left(\Omega;S_0\right); \Qvec = \sqrt{\frac{3}{2}}\left(\nvec_b \otimes \nvec_b - \frac{\mathbf{I}}{3}\right) ~\text{on}~
\partial \Omega \right\}.$ Then (up to a subsequence), we have the
following results.
\begin{enumerate}
[(i)]\item  $\{\vec Q_j\}$ converges to a
 limiting harmonic map, $\Qvec^0$ defined in (\ref{eq:form}), strongly in $W^{1,2}(\Omega; S_0)$ and uniformly
  away from $\Sigma$, as $j \to \infty$.
\item \label{it:localization} Let $\Sigma_{{\eps}} = \left\{ \xvec \in \Omega:
dist\left(\xvec,\Sigma \right) < {\eps} \right\} =
\bigcup_{\xvec_i \in \Sigma} B_{{\eps}}(\xvec_i)$ where
$B_{{\eps}}(\xvec_i)$ denotes a ball of radius ${\eps}$
centered at $\xvec_i$, and
 $ B_\delta^j =
\left\{ \xvec \in \Omega: \beta^2\left( \Qvec_j(\xvec) \right) >
\delta \right\}$, for a fixed ${\eps}>0$ and $\delta \in
\left(0, 1 \right)$. Then $B_\delta^j \subseteq \Sigma_{{\eps}}$
for $j$ large enough.
\item $\left| \Qvec_j \right| \to 1 $
uniformly on $\Omega$ as $j \to \infty$.
 \item \label{it:uniaxial1} For each $\xvec_i \in \Sigma$, we have for $j$ large
enough, 
\begin{equation}
\label{eq:thm2a} \min_{\xvec \in \overline{B_{{\eps}}(\xvec_i)}}
\beta^2\left(\Qvec_j(\xvec) \right) =0, \quad \max_{\xvec
\in \overline{B_{{\eps}}(\xvec_i)}}\beta^2\left(\Qvec_j(\xvec)
\right) = 1,
\end{equation}
and $\mathcal L^n(\{\vec x \in \overline{B_{{\eps}}(\xvec_i)}: 
	\beta^2\left (\vec Q_j(\vec x)\right )=0\})=0$.
\item For each $\xvec_i \in \Sigma$ and $\delta \in \left(0, 1 \right)$, we have
\begin{equation}
\label{eq:thm2b} \textrm{diam}\left( B_{\eps}(\xvec_i) \cap
B_\delta^j \right) \sim t_j^{-1/4}
\end{equation}
for $j$ sufficiently large.
\end{enumerate}
\end{theorem}
An immediate consequence is that global energy minimizers cannot be purely uniaxial, as also stated in \cite{contreraslamy} where the authors prove the existence of at least a single point of maximal biaxiality for global LdG minimizers.

\begin{theorem} \label{th:radial}
    Let $\Omega\subset \R^3$ be as above.
    Let $\{t_j\}_{j\in\N}$ be such that $t_j \to \infty$ as $j\to\infty$. For each $j\in\N$, let $\{\vec Q_j\}_{j\in \N}$ be a global minimizer of the LdG energy (\ref{LdGnew2}) in the restricted class of uniaxial $\Qvec$-tensors of the form (\ref{eq:uniaxial}). Then $\vec Q_j$ has non-negative scalar order parameter and
         $\vec Q_j$ satisfies the following energy bound
        \begin{equation}
        \label{eq:MAIN} \displaystyle
            \frac{3\bar L }{2Ls_+^2} \mathbf{I}_{LG}[\vec Q_j]\leq \frac{3\bar L}{2}\cdot \inf_{\vec n \in \mathcal A_{\vec n_b}} I[\vec n],
            \end{equation} with $I[\vec n]$ and ${\cal A}_{\vec n_b}$ as in \eqref{eq:OFenergy}, for each $j \in \N$. For $j$ sufficiently large, $\vec Q_j$ cannot be a stable critical point of the LdG energy in (\ref{LdGnew2}) in the admissible space,  $\bar{\Acal}_A = \left\{ \Qvec \in
W^{1,2}\left(\Omega;S_0\right); \Qvec = \sqrt{\frac{3}{2}}\left(\nvec_b \otimes \nvec_b - \frac{\mathbf{I}}{3}\right) ~\text{on}~
\partial \Omega \right\}.$
\end{theorem}
%
Theorem~\ref{th:radial} is a consequence of
Proposition~\ref{prop:radial2} below and Propositions $3$ and $8$
of \cite{hm}.
\begin{proposition}
\label{prop:radial2}
  Let $\Omega\subset \R^3$ be as above.
    Let $\{t_j\}_{j\in\N}$ be such that $t_j \to \infty$ as $j\to\infty$. Let $\vec Q_j$ be a sequence of uniaxial critical points of the re-scaled LdG energy in (\ref{LdGnew2}) with non-negative scalar order parameter and satisfying the energy bound (\ref{eq:MAIN}) for all $j >0$. Then, passing to a subsequence (still indexed by $j$),
    the sequence $\{\vec Q_j\}$ converges uniformly to a (minimizing) limiting harmonic map,
    $\Qvec^0$ as $j\to\infty$, everywhere away from the
     singular set $\Sigma = \left\{ \xvec_1 \ldots \xvec_N\right\}$ of $\Qvec^0$. We have that
    \begin{enumerate}[(i)] \item for each $i=1,\ldots, N$, there exists $\{\vec x_i^{(j)}\}_{j\in \N}$ such that $\vec Q_{j}(\vec x_i^{(j)})=\vec 0$ for all $j\in\N$ and $\vec x_i^{(j)} \overset{j\to\infty}{\longrightarrow} \vec x_i$ and
           \item given any sequence, $\{\vec x^{(j)}\}_{j\in \N}\subset \Omega$, such that $\vec Q_j(\vec x^{(j)})=\vec 0\ \forall\,j\in\N$, there exists a subsequence $\{j_k\}_{k\in \N}$ and an orthogonal transformation $\vec T\in \mathcal O(3)$ (which may depend on the subsequence) such that
        the shifted maps $\displaystyle \left\{\tilde {\vec x} \mapsto \vec Q_{j_k} \left (\vec x^{(j_k)}+\xi_b \tilde {\vec x}\right )\right\}_{k\in \N}$
        converge to
        \begin{align} \label{eq:hedgehog}
            \vec H_{\vec T}(\tilde{\vec x}):=\sqrt{\frac{3}{2}} h(|\tilde {\vec x}|) \left ( \frac{\vec T \tilde{\vec x} \otimes \vec T\tilde{\vec x}}{|\tilde {\vec x}|^2} - \frac{\vec I}{3} \right ), \quad \tilde{\vec x} \in \R^3,
        \end{align}
         in $C^r_{\loc}(\R^3; S_0)$ for all $r\in \N$, where $h:[0,\infty)\to \R^+$ is the unique, monotonically increasing solution, with $r=|\tilde{\vec x}|$, of the boundary-value problem
        \begin{align}
            \label{eq:RH}
             \frac{d^2 h}{d r^2} + \frac{2}{r}\frac{dh}{dr} -
\frac{6 h}{r^2} = h^3 - h,
\qquad h(0)=0, \qquad \lim_{r\to \infty} h(r)=1.
        \end{align}
    \end{enumerate}
\end{proposition}
Proposition~\ref{prop:radial2} provides a local description of the structural profile near a set of isotropic points in the uniaxial critical points, $\Qvec_j$, in terms of the well-known RH solution. The RH solution is a rare example of an explicit critical point of the LdG energy simply given by
\begin{equation}\label{eq:Rhedgehog}
\vec H(\tilde{\vec x}):=\sqrt{\frac{3}{2}} h(|\tilde {\vec x}|) \left ( \frac{ \tilde{\vec x} \otimes \tilde{\vec x}}{|\tilde {\vec x}|^2} - \frac{\vec I}{3} \right ), \quad \tilde{\vec x} \in \R^3,
\end{equation}
where $h$ is defined as in (\ref{eq:RH}). The boundary-value problem (\ref{eq:RH}) has been studied in detail, see for example in \cite{lamy,am2}. The RH solution is locally unstable with respect to biaxial perturbations, as has been demonstrated in \cite{ejam,g2}. 

\section{Proof of the theorems}
\label{sec:3}
%


 Recall that the re-scaled LdG energy is given by
    \begin{equation}
    \label{LdGnew}
            \frac{3\bar L}{2Ls_+^2}
        \mathbf{I}^j_{LG}[\Qvec]
        = \int_{\Omega}
        \frac{\bar{L}}{2}|\grad \Qvec|^2
        + \bar L f(\vec Q, t_j)~dV,
        \end{equation}
        with        
        \begin{equation}
\label{eq:maj6b} \bar{L} f(\Qvec,t) =
\frac{t}{8}\left(1 - |\Qvec|^2 \right)^2 +
\frac{h_+}{8}\left( 1+ 3|\Qvec|^4 - 4\sqrt{6}\textrm{tr}\Qvec^3
\right),
\end{equation}
and that for all $t>0$ the potential $f(\Qvec, t)$ is minimized on the set
\begin{equation}
\label{bpnew} \Qvec_{\min} = \left\{ \sqrt{\frac{3}{2}}\left(
\nvec \otimes \nvec - \frac{\mathbf{I}}{3} \right) :
\vec n \in \S^2 \right\}.
\end{equation}

Denote the LdG energy density by
\begin{equation}
\label{eq:maj6} e(\Qvec,t) =\frac{1}{2}|\grad \Qvec|^2 + f(\Qvec,
t).
\end{equation}
 
 In Theorem \ref{th:biaxial2}
we consider global minimizers 
$\Qvec_j$  of $\mathbf{I}^j_{LG}$
in the
        admissible space $\bar{\Acal}_A = \left\{ \Qvec \in
W^{1,2}\left(\Omega;S_0\right); \Qvec = \sqrt{\frac{3}{2}}\left(\nvec_b \otimes \nvec_b - \frac{\mathbf{I}}{3}\right) ~\text{on}~
\partial \Omega \right\}$ for each $t_j>0$,
the existence of which is guaranteed
by the direct method
        of the calculus of variations. 
        Standard elliptic regularity arguments
        (presented in \cite[Prop.~13]{amaz})
        show that each minimizer $\vec Q_j$
        is a real analytic solution 
        of the 
Euler-Lagrange equations
       \begin{equation}
\label{eq:maj13} \triangle \Qvec_{ij} = \Gamma_{ij},
\end{equation}
where
$$ \bar{L}\Gamma_{ij} = \frac{t}{2}\Qvec_{ij}\left(|\Qvec|^2 - 1 \right)
+ \frac{3 h_+}{2}\left[|\Qvec|^2 \Qvec_{ij} -
\sqrt{6}\Qvec_{ip}\Qvec_{pj} + \sqrt{6}|\Qvec|^2\delta_{ij}/3
\right]. $$
Theorem \ref{th:radial}
is proved by assuming 
        that a sequence $\{\vec Q_j\}_{j\in\N}$ of minimizers
        in the restricted class of uniaxial $\vec Q$-tensors 
        is composed of stable critical points of the LdG energy and then reaching a contradiction.
        In both cases, we consider classical solutions of \eqref{eq:maj13}
        that satisfy the energy bound \eqref{eq:MAIN}
        (this follows from the fact that
        any minimizing limiting harmonic map $\Qvec^0$
        belongs to $\bar{\mathcal{A}}_A$, so it can be used as a trial function).
        As done in \cite{hm,contreraslamy},
        the arguments in \cite[Lemmas 2 and 3; Props.~3, 4, and 6]{amaz}
        can be adapted to prove the following preliminary results.
        
        \begin{proposition} \label{pr:before}
        		Let $t_j\to +\infty$ and, for each $j\in \N$, let $\vec Q_j\in \bar{\Acal}_A$ be a classical solution 
		of the corresponding equations \eqref{eq:maj13},
		satisfying the energy bound \eqref{eq:MAIN}.
		Then, passing to a subsequence, 
		\begin{enumerate}[(i)]
			\item $\{\vec Q_j\}_{j\in \N}$ converges strongly
			 to a
			(minimizing) limiting harmonic map $\vec Q^0$ in $W^{1,2}(\Omega; S_0)$,
			\item  $\| \Qvec_j \|_{L^{\infty}} \leq 1$
			and  $\| \grad \Qvec_j \|_{L^{\infty}} \leq C \sqrt{\frac{t_j}{\bar{L}}}$
			for some  $C$ independent of $j$,
			
			\item \label{it:monotone}
			$\displaystyle \frac{1}{r}\int_{B(\vec x,r)} e(\vec Q_j, t_j)~dV
			\leq \frac{1}{R}\int_{B(\vec x, R)} e(\vec Q_j, t_j)~dV$
			for all $\vec x \in \Omega$ and $r\leq R$ so that $B(\vec x, R)\subset \Omega$,
			
			\item for any compact $K\subset \overline \Omega \setminus \Sigma$,
			where $\Sigma$ denotes the singular set of $\Qvec^0$,
			\begin{align}
			\label{eq:step1}
				 \frac{1}{8}\left(1 - |\Qvec_j|^2 \right)^2 + \frac{h_+}{8 t}( 1 + 3|\Qvec_j|^4 - 4 \sqrt{6}\tr\Qvec_j^3) \to 0
			\end{align} 
			uniformly in $K$.
		\end{enumerate}
        \end{proposition}

However, this only ensures that $|\Qvec_j| \to 1$ uniformly
        as $j \to\infty$, away from $\Sigma$. 	We want to prove the following stronger result.
	\begin{proposition} \label{pr:uniC}
		Under the hypotheses of Proposition \ref{pr:before},
		$\{\vec Q_j\}_{j\in \N}$ converges uniformly to $\vec Q^0$ away from $\Sigma$,
		as $t_j\to\infty$.
	\end{proposition}

	The key step is to prove a Bochner inequality of the form 
	\begin{lemma} \label{le:bochner}
		There exist $\eps_1>0$ and a constant $C>0$ independent of $t$
		such that if $\vec Q\in C^3(\Omega; S_0)$ is a solution of 
		\eqref{eq:maj13}
		then	
			\begin{equation}
		\label{eq:maj8} -\triangle e\left(\Qvec, t
		\right)\left(\xvec\right) \leq C e^2\left(\Qvec, t
		\right)\left(\xvec\right)
		\end{equation}
		for all $\vec x\in \Omega$ satisfying
\begin{equation}
\label{eq:maj10} 1- \eps_1 \leq |\Qvec(\vec x)| \leq 1.
\end{equation}
	\end{lemma}
	
	This  inequality was proven in \cite[Lemmas 5--7]{amaz}
	in the case when $\Qvec_j$ is close to the manifold $\Qvec_{\min}$,
	defined in (\ref{bpnew}), which does not necessarily hold in our case as explained in detail after the proof of Lemma 3.3.

\begin{proof}[Proof of Lemma \ref{le:bochner}]
 The same proof of \cite[Lemma 5]{amaz} shows that there exists a positive constant $\eps_0>0$ such that:
 \begin{eqnarray}
\label{eq:verbatim1} \frac{t}{C}f(\Qvec, t) \leq
|\Gamma|^2(\Qvec, t) \leq C t f(\Qvec, t)
\end{eqnarray}
for all $\Qvec \in S_0$ such that $\left| \Qvec -
\sqrt{\frac{3}{2}} \left(\nvec \otimes \nvec -
\frac{\mathbf{I}}{3} \right) \right| \leq \eps_0$ for some $\nvec
\in \S^2$, the positive
constant $C$ being independent of $t$.
Let $\eps_1$ be a
 positive constant (depending only on $C$ and $\eps_0$ above) such that
		\begin{equation}
\label{eq:maj10b}
0\leq |\Qvec|^3 - \sqrt{6}\tr\Qvec^3 \leq \epsilon_1
\end{equation}
and \eqref{eq:maj10}
collectively ensure that $\left| \Qvec -
\sqrt{\frac{3}{2}} \left(\nvec \otimes \nvec -
\frac{\mathbf{I}}{3} \right) \right| \leq \eps_0$
for some $\vec n\in \S^2$.

	Such an $\eps_1$ exists because
	the biaxiality parameter (see \eqref{eq:maj9})
	\begin{align}
		\beta^2(\vec Q) &=
		 \frac{|\vec Q|^3-\sqrt{6}\tr \vec Q^3}{|\vec Q|^3}
		 \frac{|\vec Q|^3+\sqrt{6}\tr \vec Q^3}{|\vec Q|^3}
		 \\ &\leq \frac{|\vec Q|^3-\sqrt{6}\tr \vec Q^3}{(1-\eps_1)^3}
		 \cdot \frac{|\vec Q|^3 + |\vec Q|^3}{|\vec Q|^3}
		\leq \frac{\eps_1}{(1-\eps_1)^3}\cdot 2
		\overset{\eps_1\to 0}{\longrightarrow} 0.
	\end{align}


The quantity $ \left|\Qvec \right|^3 -
\sqrt{6} \textrm{tr}\Qvec^3$ plays an important role in
the following proof and we note the following elementary
inequality
\begin{equation}
\label{eq:maj11} 0 \leq \left( \left|\Qvec \right|^3 - \sqrt{6}
\textrm{tr}\Qvec^3 \right) \leq \frac{\left( 3 - \sgn
\textrm{tr}\Qvec^3 \right)}{2} |\Qvec|^3.
\end{equation}

As in \cite{amaz}, we use the Euler-Lagrange equations (\ref{eq:maj13}) to derive the following inequality:
\begin{equation}
\label{eq:maj14} -\triangle e\left(\Qvec, t \right) +
|\Gamma|^2 \leq -
 2\frac{\partial^2 f}{\partial
\Qvec_{ij}\partial \Qvec_{pq}}\Qvec_{ij,k} \Qvec_{pq,k}.
\end{equation}
Moreover, we have
\begin{eqnarray}
\label{eq:maj15}&& \bar{L}^2|\Gamma|^2 =
\frac{t^2}{4}|\Qvec|^2\left(1 - |\Qvec|^2 \right)^2 + \frac{9
h_+^2}{4}(1 - |\Qvec|)^2 |\Qvec|^4 + \nonumber \\ && + \frac{3 h_+
t}{2}(1 - |\Qvec|^2)(|\Qvec|^3 - |\Qvec|^4) +\frac{3 h_+
t}{2}(|\Qvec|^2 - 1)(|\Qvec|^3 - \sqrt{6}\textrm{tr}\Qvec^3) +
\nonumber \\ && + \frac{9h_+^2}{2}(|\Qvec|^3 -
\sqrt{6}\textrm{tr}\Qvec^3)|\Qvec|^2
\end{eqnarray}
and
\begin{eqnarray}
\label{eq:maj16} &&- \bar{L}\frac{\partial^2 f}{\partial
\Qvec_{ij}\partial\Qvec_{pq}}\Qvec_{ij,k} \Qvec_{pq,k} =
\frac{t}{2}|\grad\Qvec|^2(1 - |\Qvec|^2)
-t(\Qvec\cdot\grad\Qvec)^2 - \nonumber \\ &&
- 3 h_+\left(\Qvec\cdot \grad \Qvec \right)^2 - \frac{3 h_+}{2}\left|\Qvec \right|^2 \left|\grad \Qvec \right|^2 + 3\sqrt{6} h_+ \Qvec_{\beta j}\Qvec_{\alpha
j,k}\Qvec_{\alpha\beta,k} .
\end{eqnarray}
We consider three separate cases according to the sign of
$\textrm{tr}\Qvec^3$ and the magnitude of $\left|\Qvec
\right|^3 - \sqrt{6}\textrm{tr}\Qvec^3$.

\textbf{Case I:} $0\leq \left|\Qvec \right|^3 -
\sqrt{6}\textrm{tr}\Qvec^3  \leq \eps_1$. This, when combined with
(\ref{eq:maj10}), implies that $\textrm{tr}\Qvec^3>0$ and that  $\left| \Qvec -
\sqrt{\frac{3}{2}} \left(\nvec \otimes \nvec -
\frac{\mathbf{I}}{3} \right) \right| \leq \eps_0$ for some $\nvec\in S^2$ 
(by definition of $\eps_1$).
In this case, we can repeat
all the arguments in \cite[Lemmas $5$-$7$]{amaz};
we state the key steps for completeness. 

We start with inequality (\ref{eq:verbatim1}) above.
 We denote the eigenvectors of $\Qvec$ by $\nvec_1,
\nvec_2, \nvec_3$ respectively and let $\la_3>0$ and $\la_1,
\la_2$ denote the corresponding eigenvalues. Define
$$ \Qvec^* = \sqrt{\frac{2}{3}}\nvec_3 \otimes \nvec_3 - \sqrt{\frac{1}{6}}\left(\nvec_1 \otimes \nvec_1 + \nvec_2\otimes \nvec_2 \right). $$
From the inequality $\left| \Qvec -
\sqrt{\frac{3}{2}}\left(\nvec \otimes \nvec - \frac{\mathbf{I}}{3}
\right) \right| \leq \eps_0$ with $\nvec=\nvec_3$, we necessarily
have that
$$ \left(\la_1 + \sqrt{\frac{1}{6}}\right)^2 + \left(\la_2 + \sqrt{\frac{1}{6}}\right)^2 + \left(\la_3 - \sqrt{\frac{2}{3}}\right)^2 \leq \eps_0^2. $$

The proof of the Bochner inequality now follows from the chain of inequalities below:
\begin{eqnarray}
\label{eq:verbatim2} &&-\triangle e(\Qvec, t) + |\Gamma|^2
\leq -
 2\frac{\partial^2 f}{\partial
\Qvec_{ij}\partial \Qvec_{pq}}\Qvec_{ij,k}
\Qvec_{pq,k} \leq \nonumber \\ &&
\leq \delta \sum_{i,j,m=1}^{3}
\left( \frac{\partial^3 f}{\partial \Qvec_{ij}\Qvec_{pq}\partial \Qvec_{mn} }
\left(\Qvec^*\right) \right)^2 \left(\Qvec - \Qvec^*\right)^2 
\\ &&  \hspace{5em}+
\delta \sum_{i,j,m=1}^{3} \left(\mathcal{R}^{ijmn} \right)^2 \left(\Qvec, \Qvec^* \right) + \frac{1}{\delta}|\grad \Qvec|^4 \leq \nonumber \\
&& \leq C_1 \delta t^2 \left| \Qvec - \Qvec^* \right|^2 +
\frac{1}{\delta}|\grad \Qvec|^4 \leq \nonumber \\ && \leq C_2
\delta t f(\Qvec, t) + \frac{1}{\delta}|\grad \Qvec|^4 .
\end{eqnarray}
In Equation~(\ref{eq:verbatim2}) above, we have carried out a
Taylor series expansion of the right-hand side of (\ref{eq:maj14})
about $\Qvec^*$, $\left(\mathcal{R}^{ijmn} \right)$ is the
remainder term in the Taylor series expansion which is
well-controlled and the constants $C_1$ and $C_2$ are independent
of $t$ but dependent on $\bar{L}$ (which does not matter since
$\bar{L}$ is fixed). For $\delta$ sufficiently small, we can
absorb the $ C_2 \delta t f(\Qvec, t)$-term on the right by the
$\left|\Gamma \right|^2(\Qvec)$-contribution on the left so that
$$ \left|\Gamma \right|^2(\Qvec_j) - C_2 \delta t f(\Qvec, t) \geq 0$$ for
$\delta$ sufficiently small (from (\ref{eq:verbatim1})), yielding the Bochner inequality
$$ -\triangle e(\Qvec, t)  \leq \frac{1}{\delta}|\grad \Qvec|^4$$
for $\delta>0$ independent of $t$, as required. \vspace{.5 cm}

\textbf{Case II:} $\textrm{tr}\Qvec^3>0$ and $\eps_1
<\left|\Qvec \right|^3 - \sqrt{6}\textrm{tr}\Qvec^3  \leq 1 $.

We refer to the relations (\ref{eq:maj14})-(\ref{eq:maj16}) and use the Cauchy-Schwarz inequality in (\ref{eq:maj15}) to see
that
$$ \frac{3 h_+ t}{2}(|\Qvec|^2 - 1)(|\Qvec|^3 -
\sqrt{6}\textrm{tr}\Qvec^3) \geq - \delta t^2(|\Qvec|^2 - 1)^2
- \frac{9 h_+^2}{16}\frac{1}{\delta}(|\Qvec|^3 -
\sqrt{6}\textrm{tr}\Qvec^3)^2. $$ For $\frac{3}{16} < \delta <
\frac{1-2\eps_1}{4}$ and $\eps_1$ chosen as above, we
have
\begin{eqnarray}
\label{eq:maj17} && \bar{L}^2|\Gamma|^2 \geq \alpha t^2
|\Qvec|^2(1 - |\Qvec|)^2 + \frac{9 h_+^2}{4}(1 -
|\Qvec|)^2|\Qvec|^4 + \nonumber
\\ && + \frac{3 h_+ t}{2} |\Qvec|^3 (1 - |\Qvec|)^2 (1 + |\Qvec|)
+ \eta h_+^2(|\Qvec|^3 - \sqrt{6}\textrm{tr}\Qvec^3)
\end{eqnarray}
for positive constants $\alpha, \eta$ independent of $t$ and $\bar{L}$ is fixed for our purposes.
Finally, we appeal to (\ref{eq:maj16}) to obtain
\begin{eqnarray}
\label{eq:maj18} && - \bar{L}\frac{\partial^2 f}{\partial
\Qvec_{ij}\partial \Qvec_{pq}}\Qvec_{ij,k} \Qvec_{pq,k}\leq \delta_1 \frac{t^2}{4}(1 - |\Qvec|^2)^2
 + \nonumber \\ && + \delta_2
h_+^2 |\Qvec|^2 + \frac{1}{\delta_5(\delta_1, \delta_2)}|\grad
\Qvec|^4.
\end{eqnarray} For $\delta_2$ sufficiently small, we can absorb
the $h_+^2|\Qvec|^2$ term in (\ref{eq:maj18}) by the $\eta
h_+^2(|\Qvec|^3 - \sqrt{6}\textrm{tr}\Qvec^3)$ term in
(\ref{eq:maj17}). Choosing $\delta_1,\delta_2$ small enough (and
independent of $t$), recalling
(\ref{eq:maj14}), the lower bound (\ref{eq:maj17}) and the upper
bound (\ref{eq:maj18}), we have
\begin{eqnarray}
\label{eq:maj19} -\triangle e\left(\Qvec, t \right) \leq
\frac{1}{\delta_5} |\grad \Qvec|^4
\end{eqnarray}
which is precisely the Bochner inequality.

\textbf{Case III:} $\textrm{tr} \Qvec^3 \leq 0$ so that $ \left(1-
\eps_1\right)^3 \leq \left|\Qvec \right|^3 - \sqrt{6}\textrm{tr}\Qvec^3
\leq 2 |\Qvec|^3$.

A large part of the computations for Case II carry over to Case III. In particular, (\ref{eq:maj18}) is unchanged and it remains to note that for $\textrm{tr}\Qvec^3 < 0$, the bulk potential $\bar{L}f(\Qvec,t) \geq \frac{t}{8}\left(1 - |\Qvec|^2 \right)^2 + \frac{h_+}{8}$.
In particular,
\begin{equation}
\label{eq:maj20} \frac{h_+^2}{64{\bar L}^2} \leq e^2 (\Qvec,t).
\end{equation}

Define $\sigma$ and $\gamma$ to be
\begin{eqnarray}
&& 1 - |\Qvec|^2 = \sigma \frac{h_+}{t} \nonumber
\\ && \gamma = |\Qvec|^3 - \sqrt{6} \textrm{tr}\Qvec^3
\label{eq:maj22}
\end{eqnarray}
where $\textrm{tr}\Qvec^3 \leq 0$ by assumption. 
The second, third and fifth terms in \eqref{eq:maj15} are positive, hence
\begin{eqnarray}
\label{eq:maj25}
&& {\bar L}^2 |\Gamma|^2
        \geq
\frac{|\vec Q|^2}{4} \sigma^2 h_+^2 - \frac{3h_+^2}{2}\sigma \gamma \geq \nonumber \\ && 
        \geq
\frac{h_+^2}{8} (\sigma^2-12\sigma \gamma)
        \geq
\frac{h_+^2}{8} ( (\sigma - 6\gamma)^2 - 36 \gamma^2)
        \geq
-\frac{9h_+^2}{2} \gamma^2.
\end{eqnarray}

Since $\gamma = |\vec Q|^3 - \sqrt{6}\tr \vec Q^3  \leq 2$, 
we get ${\bar L}^2 |\Gamma|^2 \geq -18h_+^2$ and therefore,  the Bochner inequality (\ref{eq:maj8})
then follows from (\ref{eq:maj14}), (\ref{eq:maj25}), (\ref{eq:maj20}).

\end{proof}

%
\textbf{Comment:}  The Bochner-inequality for the Landau-de Gennes
energy density was derived in \cite{amaz}, for global LdG
minimizers, away from the singular set, $\Sigma$, of a limiting
harmonic map, in the vanishing elastic constant limit i.e. in the
$L \to 0$ limit. There is an important mathematical difference
between the $L \to 0$ limit and the low-temperature limit ($A \to
-\infty$) considered in our manuscript. As $L \to 0$, we can use
the monotonicity formula  \cite[Lemma $2$ and Proposition~$4$]{amaz} to deduce that a global LdG minimizer, denoted by
$\Qvec_L$, is ``almost'' uniaxial with unit norm away
from $\Sigma$ as $L \to 0$ i.e.\ they are close to the uniaxial
manifold $\left\{ \sqrt{\frac{3}{2}}\left(\mvec\otimes \mvec -
\frac{\mathbf{I}}{3}\right) \right\}$ for some arbitrary $\mvec \in \S^2$. In \cite{amaz}, the authors use this proximity to the
uniaxial manifold, away from $\Sigma$, to derive the
Bochner-inequality which, in turn, yields uniform convergence to a
(minimizing) limiting harmonic map, away from $\Sigma$.

In the low-temperature limit ($A \to -\infty$), we can only prove that a
global LdG minimizer, denoted by $\Qvec_A$, satisfies $\left|
\Qvec_A \right| \to 1$ uniformly away from $\Sigma$, without any
information about uniaxiality or biaxiality (since the prefactor $\frac{h_+}{t}$ of the second term in \eqref{eq:step1}
vanishes as $t\to\infty$).
 This is weaker
information than what is available as $L \to 0$. We use the
information about $|\Qvec_A|$ as $A \to -\infty$ to derive the
Bochner inequality and this requires us to
consider three separate cases, depending on $\textrm{tr}\Qvec_A^3$
and the degree of biaxiality. Once we derive the Bochner inequality for the energy density, we can prove uniform convergence of a sequence of global energy minimizers to a limiting minimizing harmonic map, away from $\Sigma$.

From the maximum principle (see \cite{ejam})
and the uniform convergence $|\vec Q_j|\to 1$ away from the singularities of $\vec Q^0$ (see Proposition~\ref{pr:before}),
we see that (\ref{eq:maj10}) is satisfied for all $t$ sufficiently large, 
so we obtain Bochner's inequality away from $\Sigma$ for large $t$.
This enables us to deduce the following  
$\epsilon$-regularity property,
exactly as in \cite[Lemma 7]{amaz}:

\begin{lemma}
\label{epsregularity}
 Let $K\subset \Omega$ be a compact subset that does not contain any singularity of $\Qvec^0$. Then there exist $j_0$ and constants $C_1, C_2>0$ (independent of $j$) so that if for $\mathbf{a} \in K$ and $0< r <\dist\left(\vec a, \partial K \right)$, we have
\begin{equation}
\label{eq:smallenergy1}
\frac{1}{r} \int_{B(\vec a,r)}e\left(\Qvec_j, t_j \right)~dV \leq C_1,
\end{equation} then
\begin{equation}
\label{eq:smallenergy2}
 r^2 \sup_{B(\vec a, r/2)} e\left(\Qvec_j, t_j \right) \leq C_2
 \end{equation}
for all $j\geq j_0$.
\end{lemma}

\begin{proof}[Proof of Proposition \ref{pr:uniC}]
The normalized energy, $\frac{1}{r} \int_{B(a,r)}e\left(\Qvec_j,
t_j \right)~dV$, 
can be controlled away from $\Sigma$, by simply (i) using the
strong convergence of the sequence, $\left\{\Qvec_j \right\}$ to
$\Qvec^0$ as $j \to \infty$ in $W^{1,2}$ and (ii) the fact that
$|\grad \Qvec^0|$ is bounded away from $\Sigma$, independently of
$t_j$. Thus, the uniform convergence, $\Qvec_j \to \Qvec^0$, away from
$\Sigma$ as $j \to \infty$,  follows immediately from Lemma \ref{epsregularity}, combining
(\ref{eq:smallenergy1}) and (\ref{eq:smallenergy2}) and Ascoli-Arzel\'a Theorem.
\end{proof}

We are almost ready to prove the main theorems. It only remains 
to state an elementary result from homotopy theory
and to recall that the Landau-de Gennes energy functional has 
a Ginzburg-Landau like structure by blowing-up at scale $t^{-1/2}$ and working in the $t\to \infty$ 
limit.

\begin{lemma} \label{le:homotopy}
	Let $\vec Q^*(\vec x):= \displaystyle \sqrt{\frac{3}{2}} \left ( \vec n^*(\vec x) \otimes \vec n^*(\vec x)
		-\frac{\vec I}{3}\right)$
	for some 
		$\vec n^*\in C(\partial B; \S^2)$,
		where 	
		$B$ is a ball $B(\vec a, \eps)\subset \R^3$.
	Suppose that $\vec Q^*$ is homotopic in $C(\partial B; \vec Q_{\min})$
	(see \eqref{bpnew}) to $\vec Q|_{\partial B}$
	for some $\vec Q\in C(\overline{B}; \vec Q_{\min})$.
	Then $\deg \vec n^*=0$.
\end{lemma}

\begin{proof}
	Since $\vec Q|_{\partial B}$ has a continuous $\vec Q_{\min}$-valued
	extension inside $\overline B$, 
	it is homotopic in $C(\partial B; \vec Q_{\min})$ to the constant
	tensor $\vec Q(\vec a)$.
	Hence, combining the two homotopies, we deduce that $\Qvec^*$ is homotopic to a constant in 
	$C(\partial B; \vec Q_{\min})$.
	  
	Since $\partial B$ is simply-connected
	and $\S^2$ is a universal cover of $\vec Q_{\min}\cong \R P^2$,
	the latter homotopy lifts to $\S^2$, implying that $\vec n^* $ is homotopic to a constant in $C(\partial B;\S^2)$ and hence, $\deg \vec n^*=0$, as needed.	
\end{proof}

\begin{lemma} \label{le:blowup}
	Let $t_j\to +\infty$ and, for each $j\in \N$, let $\vec Q_j\in \bar{\Acal}_A$ be a classical solution 
		of \eqref{eq:maj13}.
	Suppose that $\vec Q_j$ converges strongly in $W^{1,2}$ 
	to a minimizing limiting harmonic map $\vec Q^0$.
	Let $\vec x_j^*$ be a sequence of points converging to some $\vec x^*$
	in the singular set $\Sigma$ of $\vec Q^0$.
	Then (up to a subsequence) the rescaled maps 
	    \begin{align} \label{eq:def-blowup}
        \xi_j:= \sqrt\frac{\bar L}{t_j},
        \qquad
        \vtilde x:= \frac{\vec x-\vec x_j^*}{\xi_j},
        \qquad
        \vtilde Q_j(\vtilde x):= \vec Q_j(\vec x_j^{*} + \xi_j \vtilde x)
    \end{align}
    	converge in $C^k_{\loc}(\R^3;S_0)$ for all $k\in \N$
	to a smooth solution of the Ginzburg-Landau equations
	$\Delta \vtilde Q= (|\vtilde Q|^2-1)\vtilde Q$, in $\R^3$, which satisfies
	the energy bound
        \begin{align}\label{eq:relevantenergy}
        \frac{1}{R}\int_{|\vtilde x|<R}
            \frac{1}{2}|\nabla {\vtilde Q}^\infty(\vtilde x)|^2
            + \frac{(1-|{\vtilde Q}^\infty|^2)^2}{8}~\dd V  \leq 12\pi\quad \forall\,R>0.
    \end{align}
\end{lemma}
\begin{proof}
The proof follows from the celebrated energy quantization result for  minimizing harmonic maps at singular points, established in \cite{bcl}:
    \begin{align} \label{eq:quanti1b}
        \lim_{r\to 0} \frac{1}{r} \int_{B(\vec x^*, r)} \frac{1}{2}|\nabla \vec n^0|^2 \dd V = 4\pi, \quad i=1,\ldots, N.
    \end{align}
    
We begin by noting that  $|\grad \Qvec^0|^2 = 3|\grad
     \nvec^0|^2$, therefore
    \begin{align} \label{eq:quanti2b}
        \frac{1}{r}\int_{B(\vec x^*, r)}
            \frac{1}{2}|\nabla \vec Q_j|^2 + f(\vec Q_j,t_j)~\dd V
            \ \overset{j\to\infty}{\longrightarrow}\
             \frac{3}{r}\int_{B(\vec x^*, r)} \frac{1}{2}|\nabla \vec n^0|^2~\dd V
    \end{align}
    for every small $r>0$.

    By the monotonicity formula, Proposition \ref{pr:before}\,(\ref{it:monotone}),
    for every fixed $R>0$, every small $r>|\vec x_j^{*}-\vec x^*| + \xi_j R$,
    and every $j$ sufficiently large, we have that
    \begin{align}
        \frac{1}{R}&\int_{|\vtilde x|<R}
        \frac{1}{2}|\nabla \vtilde Q_j(\vtilde x)|^2
        + \frac{(1-|\vtilde Q_j|^2)^2}{8}
        \dd V
        \\ &\leq
        \frac{1}{\xi_j R} \int_{|\vec x-\vec x_j^{*}|<\xi_j R} \frac{1}{2}|\nabla \vec Q_j(\vec x)|^2+f (\vec Q_j(\vec x), t_j)~\dd V
        \\ & \leq
        \frac{1}{r-|\vec x_j^{*}-\vec x^*|} \int_{B(\vec x_j^{*}, r-|\vec x_j^{*}-\vec x^*|)} \frac{1}{2}|\nabla \vec Q_j(\vec x)|^2+f (\vec Q_j(\vec x), t_j)~\dd V
        \\ & \leq
        \frac{r}{r-|\vec x_j^{*}-\vec x^*|}\cdot \frac{1}{r}\int_{B(\vec x^*, r)} \frac{1}{2}|\nabla \vec Q_j(\vec x)|^2+f (\vec Q_j(\vec x), t_j)~\dd V
    \end{align}
    (we have used the inequality $\frac{t}{8\bar L}(1-|\vtilde Q_j |^2)^2 \leq f(\vtilde Q_j, t_j)$ above).
    This combined with \eqref{eq:quanti2b} and \eqref{eq:quanti1b} yields the following inequality
    \begin{align}
        \begin{aligned}
            \limsup_{j\to \infty} \frac{1}{R}\int_{|\vtilde x|<R}
            \frac{1}{2}&|\nabla \vtilde Q_j(\vtilde x)|^2
            + \frac{(1-|\vtilde Q_j|^2)^2}{8}~\dd V
            \\
            &\leq 3\left (\limsup_{r\to 0^+} \frac{1}{r}\int_{B(\vec x^*, r)} \frac{1}{2}|\nabla \vec n^0|^2~\dd V\right )\leq 12\pi
        \end{aligned}
    \end{align} for every $R>0$.

    Using the energy bound above, we can extract a diagonal subsequence, converging weakly in $W^{1,2}_{\loc}\cap L^4_{\loc}(\R^3; S_0)$, to a limit map $\vtilde Q^\infty$ satisfying the  energy bound \eqref{eq:relevantenergy}.
    One can check that  $\vtilde Q^\infty$ solves the weak form of the Ginzburg-Landau equations in $\R^3$ (write the weak form of the partial differential equations for $\vtilde Q_{j_k}$ and pass to the limit when $k\to \infty$). Standard arguments in elliptic regularity then  show that
    $\vtilde Q^\infty$ is a classical solution of the Ginzburg-Landau equations and that the diagonal subsequence converges in $\bigcap_{k\in\N} C^{k}_{\loc}$ to $\vtilde Q^\infty$. 
\end{proof}

\begin{proof}[Proof of Theorem \ref{th:biaxial2}] (i) 
	It follows from Propositions \ref{pr:before} and \ref{pr:uniC}.
	
(ii) This is an immediate consequence of the uniform convergence,
$\Qvec_j \to \Qvec^0$ as $j \to \infty$, away from the singular
set, $\Sigma=\left\{\xvec_1 \ldots \xvec_N \right\}$ of $\Qvec^0$.
$\Qvec^0$ is purely uniaxial by definition i.e.
$\beta^2\left(\Qvec^0 \right) = 0$ (see (\ref{eq:maj9}) for the
definition of the biaxiality parameter, $\beta^2(\Qvec)$). The map
$\Qvec \mapsto \beta^2\left(\Qvec \right)$ is continuous for
$\Qvec \neq 0$ and the conclusion, $B_\delta^j \subseteq
\Sigma_{{\eps}}$, follows for any fixed ${\eps}$, provided $j$ is large enough.

(iii) This can be proven as in \cite{contreraslamy}, where the
authors prove that $|\Qvec_j(\xvec)|>0$ on $\Omega$, for $j$ large
enough. 
 We argue by contradiction and we assume that there exist points $\xvec_j^* \in
\Omega$ such that $\left|\Qvec_j\left(\xvec_j^*\right)\right|\leq 1 - \eta$
(for some $\eta>0$ independent of $j$), for all $j$ in the sequence. 

In view of part (i), we may assume $\xvec_j^* \to \xvec^*$ for some $\xvec^* \in \Sigma$ and
repeat the arguments in Lemma $3.1$ and $4.1$ of \cite{contreraslamy} i.e.\ perform a blow-up analysis of the re-scaled maps, $\Qvec^*_j(\xvec) = \Qvec_j\left(\xvec_j^* + \frac{\xvec}{\sqrt{t_j}} \right)$. By Lemma \ref{le:blowup} and \cite[Lemma 3.1]{contreraslamy},
 the rescaled minimizers converge locally smoothly to a minimizer, $\Qvec_\infty \in C^2(\mathbb{R}^3; S_0)$, of the Ginzburg-Landau energy, 
\begin{equation}
\label{eq:GL}
GL(\Qvec; A)=\int_A |\nabla \Qvec|^2 +\frac{1}{4} (1-|\Qvec|^2)^2 dV
\end{equation}
 (on open sets with compact closure $\subset \mathbb{R}^3$ with respect to its own boundary conditions)  with the energy growth $GL(\Qvec;B_R(0)) = \mathcal{O} ( R)$ as $R \to \infty$. In addition we have $|\Qvec_\infty(0)| \leq 1-\eta$ because of the normalization. We can then use the same blow-down analysis as in \cite{contreraslamy} to show that $\Qvec_R(\xvec) = \Qvec_\infty\left(R \xvec\right)$ as $R\to \infty$ converges strongly in $W^{1,2}_{loc}$ to a $\S^4$-valued minimizing harmonic map, labelled by $\hat{\Qvec}_\infty$.  Indeed, one can use the well-known Luckhaus interpolation Lemma as in \cite{PanatiPisante}, Proposition 4.4, still for a sequence of functionals converging to the Dirichlet integral for maps into a manifold, showing that minimality persist in the limit and the convergence is actually strong in $W^{1,2}_{loc}$.
  
 From the monotonicity formula for the Ginzburg-Landau energy, $\hat{\Qvec}_\infty$ is a degree-zero homogeneous harmonic map, hence it is smooth away from the origin by partial regularity theory \cite{schoen}. Since the latter is constant by \cite{schoen2}, the GL minimizer $\Qvec_\infty$ is also a constant matrix of norm one from the monotonicity formula for the GL energy. Thus, $|\Qvec^*_j(0)| \to |\Qvec_\infty (0)|=1$  which yields the desired contradiction.

(iv) Let $\delta \in (0,1)$ be fixed. From part (i), (ii) and since ${\eps}>0$ is fixed and arbitrary, we
necessarily have
\begin{equation}
\label{eq:bp1} \beta^2\left(\Qvec_j \right)|_{\partial
B_{{\eps}}(\xvec_i)} \leq \delta
\end{equation}
for $j$ sufficiently large, $\xvec_i \in \Sigma$  (depending only on ${\eps}$). From (iii) above, we have that $ |\Qvec_j| \to 1$ uniformly on $ \overline{B_{{\eps}}(\xvec_i)}$ as $j \to \infty$. 

Thus, if we define the set 
\begin{align}
	\mathcal{N}_\sigma =\{ \Qvec \in S_0 \, \, s.t. \, \, \beta^2(\Qvec){\leq}\sigma \, \, \rm{and} \, \, 1-\sigma \leq |\Qvec|\leq 1 \}
	\end{align}
	 for each $0\leq \sigma<1$ and let $\delta<\sigma$, we have the following: 
\begin{itemize}
\item The restriction of $\Qvec_j$ to the boundary, $\Qvec_j \in C(\partial B_{{\eps}} (\xvec_i); \mathcal{N}_\delta)$ for $j$ large enough (depending only on ${\eps}$ (by \eqref{eq:bp1}) and $\Qvec^0 \in C(\partial B_{{\eps}} (\xvec_i); \mathcal{N}_\delta)$ (in view 
of the inclusion
$\Qvec_{min} =\mathcal{N}_0 \subset \mathcal{N}_\delta$).
\item For $\delta<\sigma<1$, the maps $\Qvec_j$ and $\Qvec^0$ are homotopic in $C(\partial B_\eps (\xvec_i); \mathcal{N}_\sigma)$
(thanks to the uniform convergence; composing pointwise with the affine homotopy in $S_0$ keeps the images inside $\mathcal{N}_\sigma$ for $j$ large enough).
\item $\mathcal{N}_\sigma \supset \mathcal{N}_0$ retracts homotopically onto $\mathcal{N}_0=\Qvec_{min}\sim \mathbb{R}P^2$ for every $\sigma<1$, see \cite[Lemma 3.10]{canevari};
see also \cite{contreraslamy}, Corollary 1.2 and Section 5 therein.
\end{itemize}

	Suppose, for a contradiction, that $\max_{\overline{B_{\eps}(\xvec_i)}} \beta^2(\Qvec_j)<1$ and let $\sigma \in (\max\{\delta,  \max_{\overline{B_{\eps}(\xvec_i)}} \beta^2(\Qvec_j)\} , 1)$.
	Then the composition of the aforementioned retraction with $\vec Q_j$
	yields a map $\vec Q_j^* \in C(\overline B_\eps; \vec Q_{\min})$
	whose trace $\vec Q_j^*|_{\partial B_\eps}$ is homotopic in 
	$C(\partial B_\eps; \vec Q_{\min})$ to $\vec Q^0|_{\partial B_\eps}$.
	By Lemma \ref{le:homotopy} we would conclude that 
	$\deg \vec n^0|_{\partial B_\eps}=0$, a contradiction with the fact that
	$\textrm{deg} ~\nvec^0|_{\partial B_\eps}=\pm1$ near each singular point, for $\eps$ small enough fixed at the beginning.

Similarly, assume that $
\min_{\overline{B_{{\eps}}(\xvec_i)}} \beta^2(\Qvec_j)>0$ for infinitely many $j$ in the sequence.
Then $\Qvec_j(\xvec)$ is purely biaxial for all $\xvec \in B_{\eps}(\xvec_i)$
(recall that there are no isotropic points from part (iii)).
Let $\vec n_j(\vec x) \in \S^2$ be the eigenvector corresponding to the maximum eigenvalue 
(which is uniquely defined up to a sign). 
Recall that $\vec Q_j$ is continuous in $\overline{B_\eps(\vec x_i)}$, hence 
 $\vec x \in \Omega \mapsto \vec n_j \otimes \vec n_j \in \R P^2$
is continuous 
(if $\vec x_k\to \vec x$ and $\vec n_j(\vec x_k) \overset{k\to\infty}{\longrightarrow}
 \vec n'$
 then clearly $\vec n'$ maximizes $\vec n \cdot \vec Q_j(\vec x) \vec n$ in $\S^2$
 and the maximal eigenvalue is simple). As a consequence, we choose $\vec n_j$ to be a continuous lifting so that $\vec n_j \in C(\overline{B_\eps(\vec x_i)}; \S^2)$.
 
 Now,  $\vec Q_j$ converges uniformly to $\vec Q^0$ on $\partial B_\eps$.
Therefore, 
$|\vec Q_j| \to 1$ and $\beta(\vec Q_j)\to 0$
uniformly on $\partial B_\eps$.
This implies that  
$\vec n_j\otimes \vec n_j\to \vec n^0\otimes \vec n^0$ uniformly on $\partial B_\eps$ 
since
\begin{align*}
	 \left |\sqrt\frac {3}{2}\left (\vec n_j\otimes \vec n_j -\frac{\vec I}{3}\right )
	 	- \frac{\vec Q_j}{|\vec Q_j|}\right |
		+ \left | \frac{\vec Q_j}{|\vec Q_j|} - \vec Q^0\right |
		\overset{j\to\infty}{\longrightarrow} 0.
\end{align*}
We conclude that $\vec Q^0|_{\partial B_\eps}$ is homotopic
to $\sqrt{\frac{3}{2}}(\vec n_j\otimes \vec n_j-\frac{\vec I}{3})$,
first in $C(\partial B_\eps, \mathcal N_\sigma)$ for some small $\sigma$
(composing pointwise with the affine homotopy in $S_0$) and 
then in $C(\partial B_\eps, \vec Q_{\min})$ (composing with the retraction from $\mathcal N_\sigma$ to $\vec Q_{\min}$). Since $\sqrt{\frac{3}{2}}(\vec n_j\otimes \vec n_j-\frac{\vec I}{3})$
has a continuous extension inside $\overline{B_\eps}$, we recall Lemma \ref{le:homotopy}
and obtain a contradiction with the fact that $\deg \vec n^0|_{\partial B_\eps}=\pm 1$
for every $\eps>0$ small enough.

The uniaxial set has zero Lebesgue-measure, as has already been  established in
\cite[Prop.~14]{amaz}.

%

(v) For each $\xvec_i \in \Sigma$ and $\delta \in \left(0, 1
\right)$ fixed, consider the biaxiality set,
$B_{\eps}(\xvec_i)\cap B_\delta^j$, around $\xvec_i$ and its
diameter, $d_j:=\textrm{diam}\left(B_{\eps}(\xvec_i)\cap
B_\delta^j \right)$. We have $d_j=o(1)$ as $j \to \infty$ from
(i) and (ii) above.

We claim that $d_j \sim t_j^{-1/4}$ as $j \to \infty$, which
follows by blowing up $\Qvec_j$, at scale $d_j$, and
excluding remaining decay rates. Firstly, let $\vec p_j, \vec q_j \in B_{\eps}(\xvec_i)\cap B_\delta^j $ such that $d_j=|\vec p_j-\vec q_j|$ and let $\hat{\xvec}_j:=(\vec p_j+\vec q_j)/2$. Clearly 
$\left(B_{\eps}(\xvec_i)\cap B_\delta^j \right) \subseteq B(\hat{\xvec}_j, \frac{3d_j}{2})$ and $\hat{\xvec}_j \to \xvec_i$ as $j \to \infty$. Then by defining $B^j:=B^j(\hat{\xvec}_j,d_j/2)$ and 
by \eqref{it:localization}, 
we immediately
have $\beta^2\left(\Qvec_j \right) \leq \delta$ on $\partial B(\hat{\xvec}_j, \frac{3d_j}{2})$, $\beta^2\left(\Qvec_j \right) = \delta$ at two antipodal points on $\partial B^j$, 
and $\displaystyle \max_{\overline{B}(\hat{\xvec}_j, \frac{3d_j}{2})}\beta^2\left(\Qvec_j \right)=1$, for $j$
large enough. 


Define $\hat{\Qvec}_j(\xvec) = \Qvec_j\left(\hat{\xvec}_j +  d_j \xvec/2 \right)$ and we get, up to a sequence of rotations which we do not specify explicitly, 
\begin{eqnarray} \label{eq:bp3}
\beta^2\left(\hat{\Qvec}_j \right)|_{\partial (\frac{3}{2}B)} \leq \delta \, , \quad \beta^2\left(\hat{\Qvec}_j (0,0,\pm1) \right)= \delta\, , \quad
\max_{\overline {(\frac{3}{2}B)}}\beta^2\left(\hat{\Qvec}_j \right) = 1
\end{eqnarray} on the unit ball $B = B\left(0,1 \right)$. The the rescaled maps $\hat{\Qvec}_j$ are defined on the family of expanding domains, $2(\Omega- \hat{\xvec}_j )/ d_j \to \Rr^3$ and are local minimizers on compact subdomains of the
functionals
\begin{equation}
\label{eq:bp2} I_j\left[\hat{\Qvec}_j \right]: = \int
\frac{\bar{L}}{2}|\grad \hat{\Qvec}_j|^2 + \frac{d^2_j}{4} \left[
\frac{t_j}{8}\left(1 - |\hat{\Qvec}_j|^2 \right)^2 +
\frac{h_+}{8}\left(1 + 3|\hat{\Qvec}_j|^2 -
4\sqrt{6}\textrm{tr}\hat{\Qvec}_j^3 \right)\right]~dV
\end{equation} with $h_+ \sim \sqrt{t_j}$ as $j \to \infty$.
Taking into account the Euler-Lagrange equations (corresponding to \eqref{eq:bp2}), we can exclude
the following regimes: (a) $d_j << t_j^{-1/2}$ since we easily
deduce that (up to subsequences) $\hat{\Qvec}_j \to \Qvec_*$ in $C^k_{loc}\left(\Rr^3
\right)$ for $k \in \mathbb{N}$ by the uniform $L^\infty$-bound and elliptic regularity. Indeed, for $d_j << t_j^{-1/2}$, the nonlinear terms in the Euler-Lagrange equations vanish as $j \to \infty$. Thus $\Qvec_* \in C^2\left(\Rr^3
\right)$ is bounded and harmonic, hence constant (of norm one from (iii) above) by Liouville's Theorem and this fact contradicts
(\ref{eq:bp3}) which holds for the limiting map $\Qvec_*$ by uniform convergence. (b) $d_j \sim t_j^{-1/2}$ ; this regime has already been discussed in item (iii) above and hence, up to a subsequence,
$\hat{\Qvec}_j \to \Qvec_{**}$ in $C^k_{loc}\left(\Rr^3 \right)$ for $k \in \mathbb{N}$. Here $\Qvec_{**}$ is a bounded
Ginzburg-Landau local minimizer on the whole of $\mathbb{R}^3$ such that $\int_{B_R}
\frac{1}{2}|\grad \Qvec_{**}|^2 + \left(1 - |\Qvec_{**}|^2 \right)^2 =\mathcal{O}  (R)$ as $R \to \infty$.
Arguing as in Lemma \ref{le:blowup} and item (iii) above, we infer that $\Qvec_{**}$ is a constant matrix of norm one, contradicting \eqref{eq:bp3} which still passes to the limit under smooth convergence and clearly cannot hold for constant maps.

 (c) $t_j^{-1/2} << d_j <<
t_j^{-1/4}$. Here (up to a subsequence), the limiting map is a local minimizer of $\int |\grad \Qvec|^2$ among
$\S^4$-valued maps. Indeed the sequence is locally bounded in $H^1_{\rm loc}(\mathbb{R}^3)$ by the monotonicity formula and hence converges weakly in $H^1_{\rm loc}$ (up to a subsequence). The limiting map is clearly $\S^4$-valued, as can be seen by applying Fatou's Lemma to \eqref{eq:bp2}. Additionally, we can prove strong convergence to the limiting map and the minimality of the limiting map, arguing as in item (iii) above, i.e. using the well-known Luckhaus interpolation Lemma as in \cite{PanatiPisante}, Proposition 4.4, for a sequence of functionals converging to the Dirichlet integral for maps into a manifold. 

%
%

Therefore, $\hat{\Qvec}_j \to \Qvec_h$ in
$H^1_{loc}\left(\Rr^3 \right)$ and $\Qvec_h \in
W^{1,2}_{loc}\left(\Rr^3, \S^4 \right)$ is a minimizing harmonic map. By the regularity theory of minimizing harmonic maps  \cite{schoen}, the map $\Qvec_h$ is smooth away from a locally finite set and indeed $\Qvec_h \in C^\infty(\mathbb{R}^3; \S^4) $ by the constancy of stable tangent maps into spheres proven in \cite{schoen2}. Further, we have the energy bound, $\int_{B_R} |\grad \Qvec_h|^2
\leq C R$ for a positive constant $C$, by the monotonicity formula which allows us to blow-down $\Qvec_h$ from infinity.
Thus, $\Qvec_h$ has minimizing tangent maps at infinity and  the rescaled harmonic maps converge strongly to the tangent maps (up to a subsequence) by Luckhaus compactness theorem for harmonic maps.
We use the constancy of stable tangent maps into spheres from \cite{schoen2} and the monotonicity formula, arguing by analogy with case (b), to infer that $\Qvec_h$ is a constant matrix of norm one.
In view of this constancy property, we can improve the convergence $\hat{\Qvec}_j \to \Qvec_h$ in
$H^1_{loc}\left(\Rr^3 \right)$ to a smooth convergence (we just need to use the argument based on the Bochner inequality from (i) above). Since biaxiality is constant for constant maps,
we contradict (\ref{eq:bp3}). 

Finally, we consider the
regime (d) $t_j^{-1/4} << d_j << 1$. Here, the limiting energy is
again the Dirichlet energy, $\int |\grad \Qvec|^2~dV$, for
$\Qvec_{min}$-valued maps in $H^1_{\rm loc}(\mathbb{R}^3)$, as can be seen by applying Fatou's Lemma to \eqref{eq:bp2}. We again have $\hat{\Qvec}_j \to
\Qvec_h$ in $H^1_{loc}\left(\Rr^3 \right)$, arguing similarly to part (c) above. However, from the uniaxiality of the limiting tensor and the lifting results in \cite{bz}, we lift $\Qvec_h$ to an
$\S^2$-valued minimizing harmonic map $\bar{\vec n} \in H^1_{loc}(\mathbb{R}^3;\S^2) $. From the classification result for harmonic unit-vector fields, such as $\bar{\vec n}$, in \cite[Thm. 2.2]{AlmgrenLieb}, we either
have $\Qvec_h = constant$ or $\Qvec_h =
\sqrt{\frac{3}{2}}\left(\frac{\xvec \otimes \xvec}{|\xvec|^2} -
\frac{\mathbf{I}}{3} \right)$, again as in step (c) with locally smooth convergence except at most at one point (combining the smoothness of the limiting map with small energy regularity to infer smooth convergence). This contradicts (\ref{eq:bp3}) since
$\beta^2\left(\Qvec_h \right) = 0$ {everywhere except  possibly for the origin, since $\Qvec_h$ is uniaxial for $\xvec \neq \mathbf{0}$.}

\end{proof}

\begin{proof}[Proof of Theorem \ref{th:radial}]

We can prove the existence of a global LdG minimizer $\Qvec_j$, of the re-scaled energy (\ref{LdGnew2}), in the restricted class of uniaxial $\Qvec$-tensors, for each $t_j$, from the direct methods in the calculus of variations. It suffices to note that the uniaxiality constraint, $6\left(\textrm{tr} \Qvec^3 \right)^2 = |\Qvec|^6$ is weakly closed and the existence result follows immediately.

The limiting harmonic map $\Qvec^0$ is uniaxial and hence, the energy bound (\ref{eq:MAIN}) follows immediately since the upper bound is simply the re-scaled LdG energy of $\Qvec^0$. The uniaxial map, $\Qvec_j =  s_j\left(\nvec_j \otimes \nvec_j - \frac{\mathbf{I}}{3} \right)$, necessarily has non-negative scalar order parameter. Indeed, note that by uniaxiality, $\det \, \Qvec(x) >0$ (resp. $\det \, \Qvec(x)<0$) iff $\Qvec(x)$ has positive (resp. negative) scalar order parameter and also that $\det \, \Qvec(x)=0$ iff $\Qvec(x) =0$ at any $x\in \Omega$. We set $\Omega_j:= \{ \det \, \Qvec_j(x)<0 \}  \subset \Omega$, which is an open subset (possibly empty), since $\Qvec_j$ is globally Lipschitz in $\Omega$. If $\Omega_j \neq \emptyset$, then we define the uniaxial admissible perturbation
\begin{equation}
 \Qvec_j^* (\rvec) = \begin{cases}
\Qvec_j \quad \rvec\in \Omega\setminus  \Omega_j \\
- \Qvec_j \quad \rvec\in \Omega_j
\end{cases}
\end{equation}
and one can easily check that $\Qvec^*_j$ is globally Lipschitz in $\Omega$ and $ \frac{3\bar L}{2Ls_+^2}
        \mathbf{I}^j_{LG}[\Qvec_j^*] < \frac{3\bar L}{2Ls_+^2}
        \mathbf{I}^j_{LG}[\Qvec_j]$, contradicting the assumed global minimality of $\Qvec_j$ in the restricted class of uniaxial $\Qvec$-tensors.
        We can then appeal to Proposition~\ref{prop:radial2} and proceed by contradiction. We assume that the global LdG-minimizers, $\Qvec_j$, in the restricted class of uniaxial $\Qvec$-tensors, are stable critical points of the LdG energy, for $j$ large enough. The sequence, $\left\{\Qvec_j \right\}$, then satisfies the hypothesis of Proposition~\ref{prop:radial2}, for large $j$. We thus, conclude that each $\Qvec_j$, has a  set of isotropic points $\xvec_i^{(j)}$ (at least one near each singular point $\xvec_i$ of $\Qvec^0$) and $\Qvec_j$ is asymptotically described by the RH-profile near each isotropic point $\xvec_i^{(j)}$ as $j \to \infty$ in the sense of 
        Proposition \ref{prop:radial2}. Recall that the RH-solution, (\ref{eq:hedgehog}) is known to be unstable with respect to biaxial perturbations localized around the origin \cite{ejam}, \cite{g2}. This suffices to prove that global minimizers in the restricted class of uniaxial $\Qvec$-tensors cannot be stable critical points of the LdG energy in the low-temperature limit, since stability of $\Qvec_j$ would pass to the limit under smooth convergence.
        \end{proof}

  \begin{proof}[Proof of Proposition~\ref{prop:radial2}]

  	{\it Proof of (i):}
	By Propositions \ref{pr:before} and \ref{pr:uniC}, after extracting a subsequence, we have that $\left\{\Qvec_j\right\}$ converges strongly in $W^{1,2}$ and uniformly away
	 from the singular set $\Sigma=\left\{\xvec_1 \ldots \xvec_N \right\}$, 
	to a (minimizing) limiting harmonic map, $\Qvec^0$.
	 We prove that for each $i=1,\ldots, N$ and every fixed $r_0>0$ sufficiently small, there exists $j_0\in \N$ such that for every $j\geq j_0$,
    the map $\vec Q_j$ has an isotropic point, $\vec x_i^{(j)}$, in $\overline B(\vec x_i, r_0)$. The stated conclusion then follows by a diagonal argument on $r_0$. Suppose, for a contradiction, that we can find a subsequence, $\{j_k\}_{k\in\N}$, such that $\min_{B(\vec x_i, r_0)} |\vec Q_{j_k}|>0$ for all $k\in\N$.
    {Since $\vec Q_j$ is purely uniaxial for all $j$ by assumption, we have that $\frac{\vec Q_{j_k}}{|\vec Q_{j_k}|}$ is continuous on $\overline{B(\vec x_i, r_0)}$ and
    the uniform convergence to $\vec Q^0$ implies that $\frac{\vec Q_{j_k}}{|\vec Q_{j_k}|}$
    converges uniformly to $\vec Q^0$ on $\partial B_\eps$.}
    Arguing as in the proof of Theorem \ref{th:biaxial2}\,\eqref{it:uniaxial1}
    we obtain a contradiction.

    {\it Proof of (ii):}
    The aim is to prove that $\vec Q_j$ has a radial-hedgehog type of profile,  (\ref{eq:hedgehog}), near each singular point in $\Sigma$, for $j$ sufficiently large. The proof follows from Lemma \ref{le:blowup} and Propositions $4$ and $8$ in \cite{hm}.
    We begin by noting that for each $i=1\ldots N$
    in $\{\vec x_1, \ldots, \vec x_N\}$, we can extract a sequence,
    $\left\{\vec x_j^*\right\}$, such that
    $\vec Q_j\left(\vec x_j^* \right) = 0$ and $\vec x_j^* \to \vec x_i$ as $j\to \infty$.
	By Lemma \ref{le:blowup}, the rescaled maps 
	\eqref{eq:def-blowup}
	converges in $\bigcap_{k\in\N} C^{k}_{\loc}$
	to a classical solution $\vtilde Q^\infty$ of the Ginzburg-Landau equations
	satisfying the energy growth \eqref{eq:relevantenergy}.
	Moreover, it can be seen that $\vtilde Q^\infty$ is uniaxial and has a  non-negative scalar order parameter.
	Finally $\vtilde Q^\infty (\vec 0)=\vec 0$ because $\vtilde Q_{j_k}(0)=\vec 0$ for each $k$, by assumption. We conclude that the hypotheses of \cite[Prop.~8]{hm} are satisfied. We reproduce the statement of \cite[Prop.~8]{hm} below, for completeness.
    \begin{proposition} [Proposition 8, \cite{hm}]
\label{prop:S}
Let $\vec Q\in C^2(\R^3; S_0)$ be a uniaxial solution of $\Delta  \Qvec = (|\Qvec|^2-1)\Qvec$ with $\vec Q(0)=0$ and non-negative scalar order parameter, satisfying
the energy bound \eqref{eq:relevantenergy}.
Let $h$ denote the unique solution for the boundary-value problem \eqref{eq:RH}.
Then there exists an orthogonal matrix $\vec T \in \mathcal O(3)$ such that
\begin{align}
\label{eq:main}
\Qvec(\vec x) = \sqrt{\frac{3}{2}} h( |\xvec|) \left( \frac{\vec T \xvec \otimes \vec T\xvec}{|\xvec|^2} - \frac{\mathbf{I}}{3} \right), \quad \xvec \in \R^3.
\end{align}
\end{proposition}
This yields the conclusion of Proposition~\ref{prop:radial2}.

\end{proof}

\section{Conclusions}
\label{sec:4}

{Theorem~\ref{th:biaxial2} focuses on global minimizers of the LdG energy on arbitrary 3D domains, with arbitrary topologically non-trivial Dirichlet conditions, for low temperatures. We prove that global minimizers are ``almost" uniaxial everywhere away from the singular set of a (minimizing) limiting harmonic map, $\Qvec^0$. Further, we prove that global minimizers have at least a point of maximal biaxiality (with $\beta^2=1$)  and a point of pure uniaxiality (with $\beta^2=0$) near each singular point of $\Qvec^0$ and their norm converges uniformly to unity everywhere, for low temperatures. This yields quantitative information about the expected number and location of points of maximal biaxiality and provides rigorous justification for the widely used Lyuksyutov constraint for the LdG energy in the low-temperature limit,  suggesting that we may be able to analytically recover the celebrated biaxial torus solution  \cite{gartland, kraljvirga, sluckin, penzenstadler} by a blow-up analysis of the LdG energy using scalings related to the decay estimate of strongly biaxial regions derived in Theorem~\ref{th:biaxial2}. Recall that the biaxial torus solution (numerically) exhibits a ring of maximal biaxiality and defect cores with uniaxial states that have negative order parameter and from the results in Theorem~\ref{th:biaxial2}, we conjecture that we may find a biaxial torus solution near each singular point of a (minimizing) limiting harmonic map.}

Theorem~\ref{th:radial} focuses on global LdG minimizers within the restricted class of uniaxial $\Qvec$-tensors for low temperatures. These constrained uniaxial minimizers exist although they need not be critical points of the LdG energy. Indeed, uniaxial critical points of (\ref{eq:ELeqs}) are, in general, difficult to find. In \cite{lamyuniaxial}, the author excludes purely uniaxial critical points of the LdG energy in 1D and 2D but the radial-hedgehog (RH) solution is a 3D uniaxial critical point of the LdG energy i.e.\ is a solution of the system (\ref{eq:ELeqs}) of the form (\ref{eq:uniaxial}) with $s>0$ for $r>0$. Indeed, one could imagine a continuous uniaxial perturbation of the RH solution that remains a solution of the system (\ref{eq:ELeqs}). An alternative scenario is that we glue together several copies of the RH solution, with a weak uniaxial perturbation of a limiting harmonic map interpolating between the distinct RH-copies, to yield an uniaxial solution of (\ref{eq:ELeqs}), at least in some approximate sense. Proposition~\ref{prop:radial2} (of which constrained uniaxial minimizers are a special case) has a two-fold purpose: (i) firstly, it rules out the stability of such uniaxial critical points, if they can be constructed and (ii) secondly and perhaps more importantly, it establishes the universal RH-type defect profiles for uniaxial critical points (if they exist) of the LdG energy for low temperatures.

We conjecture that the universal RH-type defect profile is generic for sequences of uniaxial critical points that converge strongly to a (minimizing) limiting harmonic map, under some physically relevant hypotheses which maybe different in different asymptotic limits. For example, we can consider a sequence of physically relevant uniaxial critical points in the vanishing elastic constant limit $L \to 0$.
Here, physical relevance can again be understood in terms of non-negative scalar order parameter and an appropriate energy bound. The $L \to 0$ limit has been well-studied in \cite{amaz} and we can appeal to Lemmas $6-7$ of \cite{amaz} or Case I of Theorem~\ref{th:biaxial2}[(i)] to deduce that the uniaxial sequence converges uniformly to a limiting (minimizing) harmonic map, $\Qvec^0$,  away from the singular set of $\Qvec^0$. We conjecture that one can repeat the arguments in Proposition~\ref{prop:radial2} to deduce (i) the existence of an isotropic point near each singular point, $\xvec_i$, of $\Qvec^0$ and (ii) the local RH-type defect profile near each such isotropic point.
We hope to make rigorous studies of uniaxial and biaxial defect profiles, in different temperature regimes, in future work. 


\subsection*{Acknowledgements}
    AM is supported by an EPSRC Career Acceleration Fellowship EP/J001686/1 and EP/J001686/2, an Oxford Centre for Industrial Applied Mathematics (OCIAM) Visiting Fellowship and association with the Advanced Studies Centre, Keble College, Oxford.
    DH is supported by FONDECYT project n$^\circ$ 1150038 from the Chilean Ministry of Education.
    AP thanks OCCAM for supporting his collaborative visit in March 2012.
    DH thanks OxPDE, University of Oxford, for financial support in October 2012.
    DH, AM thank the Mathematics of Liquid Crystals Research Programme for supporting their stay at the Isaac Newton Institute, University of Cambridge in April 2013.
    We thank Giacomo Canevari and Arghir Zarnescu for helpful discussions.

\bibliography{biaxial2} \bibliographystyle{plain}   
\end{document}